\documentclass[a4paper, 12pt]{amsart}   
\usepackage{mathptmx, amssymb,amscd,latexsym, eulervm}   
\usepackage{amsmath}
\usepackage{amsthm}
\usepackage{mathdots}
\usepackage{setspace}

\usepackage{hyperref}
\hypersetup{
 colorlinks,
 linkcolor={blue!90!black},
 citecolor={red!80!black},
 urlcolor={blue!50!black},
breaklinks=true
}
\usepackage{tabularx}
\usepackage{comment}
\usepackage{amsfonts}
\usepackage{paralist}
\usepackage{aliascnt}
\usepackage[initials, lite]{amsrefs}
\usepackage{amscd}
\usepackage{blkarray}
\usepackage{booktabs}
\usepackage{enumitem}
\usepackage{multicol}
\usepackage{setspace}
\usepackage[inner=2.5cm,outer=2.5cm, bottom=3.2cm]{geometry}
\usepackage{hyperref}
\usepackage[capitalize,nameinlink]{cleveref}
\usepackage{tikz, tikz-cd}
\usepackage{calligra,mathrsfs}
\usetikzlibrary{matrix}
\usetikzlibrary{arrows,calc}
\usepackage{url}

\allowdisplaybreaks

\newtheorem{thm}{Theorem}
\newtheorem{lemma}[thm]{Lemma}
\newtheorem{cor}[thm]{Corollary}

\theoremstyle{definition}
\newtheorem{example}[thm]{Example}
\newtheorem{notation}[thm]{Notation}
\newtheorem{remark}[thm]{Remark}

\newtheorem{question}[thm]{Question}

\newtheorem{conj}{Conjecture}

\newcommand{\Hom}{\mathrm{Hom}}
\newcommand{\Ext}{\mathrm{Ext}}
\newcommand{\Hilb}{\mathrm{Hilb}}

\newcommand{\gr}{\mathrm{gr}}

\renewcommand{\AA}{\mathbb{A}}

\newcommand{\N}{\mathbb{N}}
\newcommand{\Z}{\mathbb{Z}}

\newcommand{\kk}{{\Bbbk}}

\makeatletter
\@namedef{subjclassname@2020}{
  \textup{2020} Mathematics Subject Classification}
\makeatother

\begin{document}

\author[R.\,Ramkumar, A. \,Sammartano]{Ritvik~Ramkumar and Alessio~Sammartano}
\address{(Ritvik Ramkumar) Department of Mathematics\\Cornell University\\Ithaca, NY\\USA}
\email{ritvikr@cornell.edu}
\address{(Alessio Sammartano) Dipartimento di Matematica \\ Politecnico di Milano \\ Milan \\ Italy}
\email{alessio.sammartano@polimi.it}

\title{On the parity conjecture for Hilbert schemes of points on threefolds}


\begin{abstract} 
Let $\Hilb^d(\AA^3)$ be the Hilbert scheme of $d$ points in $\AA^3$, 
and let $T_z$ denote the tangent space to a point $z\in \Hilb^d(\AA^3)$.
Okounkov and Pandharipande have conjectured that  $\dim T_z$ and $d$ have the same parity for every $z$.
For points  $z$ parametrizing monomial ideals, the conjecture was proved by Maulik, Nekrasov, Okounkov, and Pandharipande.
In this paper, we settle the conjecture for points $z$ parametrizing homogeneous ideals.
In fact, we state a generalization of the problem to Quot schemes of $\AA^3$, 
and we settle it for points parametrizing graded modules.
\end{abstract}

\maketitle

\section*{Introduction}

The Hilbert scheme of $d$ points on $\AA^n$, denoted by $\Hilb^d(\AA^n)$, 
parameterizing closed zero-dimensional subschemes of degree $d$, 
was constructed by Grothendieck \cite{Grothendieck}. 
In the case when $n=2$, the Hilbert scheme is a beautiful space: 
it is a smooth variety of dimension $2d$ \cite{Fogarty}, and understanding its geometry has always been of significant interest. 
For some of the main highlights of the theory of Hilbert schemes of points in dimension two, 
and its many connections to other fields of mathematics, 
we refer the reader to \cite{Gottsche,Haiman,IarrobinoBook,NakajimaBook}.

By contrast, if $n\geq 3$, then  $\Hilb^d(\AA^n)$ becomes considerably more pathological. 
It is  singular as soon as $d \geq 4$, and it has many irreducible components if $d \gg 0$.
Moreover, recent breakthrough \cite{Jelisiejew} shows that it  can have arbitrarily bad singularities for sufficiently large $n$, in particular, it can be non-reduced.

However, when $n=3$, while still being singular and reducible, 
the Hilbert scheme 
$\Hilb^d(\AA^3)$ is rather special.
For instance,  it admits a superpotential description,
that is,  it can be realized as the critical locus of a regular function on a smooth variety \cite{BBS}.
This 
 suggests that  $\Hilb^d(\AA^3)$ may still be a well-behaved space.
Over the past two decades, there has been a lot of interest in $\Hilb^d(\AA^3)$,
especially in the field of enumerative geometry.
Of particular importance is its role in the
seminal work of Maulik, Nekrasov, Okounkov, and Pandharipande \cite{MNOP}, connecting  Gromov-Witten theory and Donaldson-Thomas theory for threefolds. 
In the process of developing this correspondence, 
 Okounkov and Pandharipande observed  the following parity phenomenon (see e.g. \cite{PandharipandeLecture}).

\begin{conj}[Okounkov-Pandharipande, 2006]\label{ConjOP}
For all $z\in \Hilb^d(\AA^3)$, we have 
$$\dim T_z \equiv d \pmod{2}.$$
\end{conj}

This is a somewhat bizarre statement:
in general, there is no reason to expect 
a    scheme  to have (regular) jumps in tangent space dimensions.
It might be a manifestation of some  hidden structure of points in $\AA^3$. 
For instance, $d$ has the same parity as $3d$, which is the dimension of the smoothable component;
thus, Conjecture \ref{ConjOP}
suggests  that the tangent spaces to points away from the smoothable component may still have some relation to the smoothable component. 
We emphasize that this is a  phenomenon special for threefolds:
clearly, the dimension of  tangent spaces  to $\Hilb^d(\AA^2)$ is constant,
whereas  tangent spaces to
$\Hilb^d(\AA^n)$ do not exhibit any  noteworthy pattern for $n \geq 4$,
see Remark \ref{RemarkExamplesHilbN}.

For points $z$ parametrizing  monomial ideals, the conjecture was proved by Maulik, Nekrasov, Okounkov, and Pandharipande in \cite[Theorem 2]{MNOP}.
Specifically,
 the monomial case was a crucial ingredient in  calculating the partition function of the Donaldson-Thomas theory of $\AA^3$ in terms of the MacMahon function for 3-dimensional partitions.
 See also  \cite[Lemma 4.1 (c)]{BehrendFantechi} and \cite[Corollary 2.5]{RS}.
 
The goal of this paper is to  settle a more general case of  the conjecture:

\begin{thm}\label{TheoremMain}
Conjecture \ref{ConjOP} holds for all points $z$ parameterizing homogenous ideals.
\end{thm}

Homogeneous ideals are much more abundant than monomial ideals. 
Clearly, there are infinitely many of  them in $\Hilb^d(\AA^3)$, and, more significantly, 
there are points parametrizing homogeneous ideals which lie outside the smoothable component of $\Hilb^d(\AA^3)$ 
\cite{Iarrobino}.
In fact, our proof gives a larger locus where the conjecture holds, as it allows for more general gradings, see Remark \ref{RemarkGradings}.

More generally, we consider an extension of Conjecture \ref{ConjOP}  to Quot schemes of $\AA^3$,
 which asks whether the dimensions of tangent spaces have the same parity as the dimension of the principal component, see
Question \ref{QueQuot}.
Again, if this were true, it would be quite surprising,
 since it would imply, for instance, that if the vector bundle has even rank, then  all the tangent spaces have even dimension, regardless of the length of the quotients. 
As mentioned before, this phenomenon does not hold for $\AA^n$ with $n \geq 4$;
however,
and more interestingly,
it also does not occur for the (singular!) Quot schemes of $\AA^2$,
see Example \ref{ExampleQuotA2}.
Our main result is Theorem \ref{TheoremParityModules}, in which we affirmatively answer
 Question \ref{QueQuot} for points parametrizing graded modules; specializing to the line bundle $\mathcal{O}_{\AA^3}$, this gives Theorem \ref{TheoremMain}.
To prove Theorem \ref{TheoremParityModules}, we compare the tangent space  to a point $z$ parametrizing a graded module 
and the tangent space to a point $z'$ 
obtained from $z$ by a Gr\"obner degeneration.
It is well-known that, by semicontinuity, we have $\dim T_z \leq \dim T_{z'}$.
We refine this inequality and prove a more precise statement:
namely, that there is a correspondence between graded components of $T$ of odd and even degree,
such that 
$T_z$ is obtained from $T_{z'}$ by canceling pairs of subspaces of equal dimension in corresponding graded components.

While the present paper was under review, the preprint \cite{GGGL} appeared, proving that Conjecture \ref{ConjOP} is false in its full generality: 
the authors exhibit an explicit $z \in \Hilb^{12}(\AA^3)$ with odd tangent space dimension.
This counterexample suggests several questions, for example: 
Does Conjecture \ref{ConjOP} hold on a dense open subset of $\Hilb^d(\AA^3)$?
Is its failure  related to the existence of non-smoothable irreducible components or  embedded components of $\Hilb^d(\AA^3)$?
What are the connections to the questions of the constancy of the Behrend function
and of the  generic reducedness of $\Hilb^d(\AA^3)$?
See \cite{Ricolfi} for some recent results in this direction.

\section*{Proof of the main result}

Let $\kk$ denote an algebraically closed field of arbitrary characteristic.
Let $X$ be a smooth threefold over $\kk$, and let $\mathcal{F}$ be a vector bundle of rank $r$ over $X$.
For a positive integer $d$,
the  {\bf Quot scheme} $\mathrm{Quot}^d_X(\mathcal{F})$ parametrizes finite length quotients of $\mathcal{F}$ of degree $d$.
There is an irreducible component of $\mathrm{Quot}^d_X(\mathcal{F})$, known as the principal component, 
whose general point parametrizes quotients of $\mathcal{F}$ that are abstractly isomorphic to 
$\bigoplus_{i=1}^d \mathcal{O}_{p_i}$, for some distinct points $p_1, \ldots, p_d \in X$.
The dimension of the principal component is $(2+r)d$.
When $r=1$ and $\mathcal{F}= \mathcal{O}_X$,
the Quot scheme coincides with the Hilbert scheme of points $\mathrm{Hilb}^d(X)$.
Given a closed point $z \in \mathrm{Quot}^d_X(\mathcal{F})$ parametrizing a quotient $\mathcal{F} \twoheadrightarrow \mathcal{M}$ with kernel $\mathcal{K}$, the tangent space to $\mathrm{Quot}^d_X(\mathcal{F})$ at the point $z$ is $T_z = \mathrm{Hom}(\mathcal{K},\mathcal{M})$.
We refer to \cite{Sernesi} for more details on Quot schemes.

We consider the following generalization of Conjecture \ref{ConjOP}:

\begin{question}\label{QueQuot}
Let $z \in \mathrm{Quot}^d_X(\mathcal{F})$. 
Is it true that
$
\dim T_z \equiv rd \pmod{2}?
$
\end{question}

In other words, Question \ref{QueQuot} asks whether the dimension of the tangent space to any point on the Quot scheme has the same parity as the dimension of the principal component.
Since $X$ is smooth, 
it suffices to answer Question \ref{QueQuot} in the affine case,
that is,
 when $X = \AA^3$ and $\mathcal{F}= \mathcal{O}_{\AA^3}^{\oplus r}$ is a trivial vector bundle.
 Clearly, Conjecture \ref{ConjOP} is the case $r=1$ of Question \ref{QueQuot}.

From an algebraic perspective,
the Quot scheme $\mathrm{Quot}_{\AA^3}^d(\mathcal{O}^{\oplus r})$ parametrizes finite length quotients of $F$ of length $d$,
where $F$ is a free module of rank $r$ over  $S= \kk[x_1, x_2, x_3]$.
Thus, for the rest of the paper, we work with modules over the polynomial ring.

We begin with a general fact about  
Ext modules and initial submodules with respect to a monomial order.
We refer to \cite{Eisenbud,BCRV} for the theory of Gr\"obner bases for modules.
We denote  by $[\cdot]_j$ the  component of degree $j$ of a graded module.

\begin{lemma}[Consecutive cancellations]\label{LemmaCancellations}
Let $S = \kk[x_1, \ldots, x_n]$ be a polynomial ring, 
equipped
with the standard grading.
Let $F$ be a free graded $S$-module, and  
 $\preceq$ a  monomial order on $F$ which refines the grading.
Let $K \subseteq F$ be a graded submodule such that  $M=F/K$ is a module of finite length.
Let $K' \subseteq F$ be the monomial initial submodule of $K$ with respect to   $\preceq$,
and let $M' = F/K'$.
Then,
there exist integers $\delta_{i,j} \in \N$ such that
$$
\dim_\kk\big[\Ext^i_S(M',M')\big]_j = \dim_\kk\big[\Ext^i_S(M,M)\big]_j + \delta_{i,j}+\delta_{i-1,j}
$$
for all $i = 0, \ldots, n$ and all  $j \in \Z$,
where  $\delta_{n,j} =\delta_{-1,j} = 0$ for all $j$.
\end{lemma}

\begin{proof}
We use $\deg(\cdot)$ to denote the grading on $S$ and $F$.
Let  $\{\varepsilon_1, \ldots, \varepsilon_r\}$ denote the basis of $F$.
Given a weight  
$\omega = (w_1, \ldots, w_n, v_1, \ldots, v_r ) \in \N_{>0}^{n+r}$,
we can  define another  grading on $S$ and $F$, which we  denote by $\omega(\cdot)$,
by setting $\omega(x_i) = w_i$ and $\omega(\varepsilon_i) = v_i$.
Both the $\deg$-grading and the $\omega$-grading  play an important role in this proof.

The weight $\omega$ defines  increasing filtrations on $S$ and $F$,
and on their submodules and quotients.
We use the symbol $\mathcal{W}_p(\cdot)$ to denote all these filtrations.
For example, we have 
$$\mathcal{W}_p S = \{ f \in S \mid \omega(f) \leq p\}
\quad 
\text{and}
\quad
\mathcal{W}_p F = \{ f \in F \mid \omega(f) \leq p\}
\quad 
\text{for } p \in \Z.
$$ 
Given an $S$-module $N$ with a filtration $\{\mathcal{W}_pN \, \mid \, p \in \Z\}$,
we denote the associated graded module by  $\gr(N) = \oplus_{p \in \Z} \mathcal{W}_p N / \mathcal{W}_{p-1}N$.
Since $S$ and $F$ are already graded with respect to $\omega$, we have $\gr(S) \cong S$ and $\gr(F) \cong F$.

Now, consider the submodule $K \subseteq F$.
Then, there exists a weight $\omega$ such that $\gr(K) = K'$,
hence $\gr(M) = M'$.
This is well-known in the case of ideals, 
see e.g. \cite[Proposition 15.16]{Eisenbud},
but it also true for submodules, see e.g. \cite[Remark 1.5.9]{BCRV}.
For the rest of the proof, we fix such a weight $\omega$.

There exists a complex $\mathbf{F}_\bullet$ of filtered $S$-modules
that is a $\deg$-graded free resolution of $M$, and such that
$\gr(\mathbf{F}_\bullet)$ is a minimal $(\deg,\omega)$-graded free resolution of $M'$.
The existence of such $\mathbf{F}_\bullet$ follows 
from standard arguments about filtered modules and standard bases,
see for example \cite[Proposition 2.3]{BHV}.
Then, 
the complex $\mathbf{L}^\bullet = \Hom(\mathbf{F}_\bullet, M)$ is also  a complex of filtered $S$-modules,
where the filtrations are defined by 
$$
\mathcal{W}_p \mathbf{L}^i = 
\big\{ f \in \mathbf{L}^i = \Hom(\mathbf{F}_i, M) \, \mid \, f(\mathcal{W}_q\mathbf{F}_i) \subseteq \mathcal{W}_{q+p} M \, 
\text{ for all } q\big\},
$$
and we have the isomorphism of $(\deg,\omega)$-graded complexes
$$
\gr(\mathbf{L}^\bullet) = \gr(\Hom(\mathbf{F}_\bullet,M)) \cong \Hom(\gr(\mathbf{F}_\bullet),\gr(M))
= \Hom(\gr(\mathbf{F}_\bullet),M').
$$

We consider the spectral sequence $\mathbf{E}^{p,q}_r$ associated to the filtered complex $\mathbf{L}^\bullet$, see \cite[\href{https://stacks.math.columbia.edu/tag/012K}{Section 12.24}]{stacks}.
It begins with page
$$
\mathbf{E}_0^{p,q} = \gr(\mathbf{L}^{p+q})_p,
$$
where $p$ is the filtration degree and $p+q$ is the total cohomological degree.
Notice that each $\mathbf{E}_0^{p,q}$ is finite dimensional vector space,
and that $\mathbf{E}_0^{p,q} \ne 0$ for finitely many pairs $(p,q)$.

We emphasize, in particular, that 
 $\gr(\mathbf{L}^\bullet)$ is a $\deg$-graded complex, and its differentials are $\deg$-homogeneous (of degree 0).
It follows that all  vector spaces $\mathbf{E}_r^{p,q}$ in the spectral sequence are $\deg$-graded, 
and that
all   differentials 
$
d_r^{p,q} : \mathbf{E}_r^{p,q} \to \mathbf{E}_r^{p-r,q+r+1}
$
are   $\deg$-homogeneous of degree 0.

The first page of the spectral sequence is 
$$
\mathbf{E}_1^{p,q} = H^{p+q}(\gr(\mathbf{L}^\bullet)_p) = H^{p+q}(\Hom(\gr(\mathbf{F}_\bullet),M')_p) = (\Ext^{p+q}(M',M'))_p,
$$
whereas the infinity page is 
$$
\mathbf{E}_\infty^{p,q} = \gr(H^{p+q}(\mathbf{L}^\bullet))_p = 
\gr(H^{p+q}(\Hom(\mathbf{F}_\bullet,M)))_p = 
\gr(\Ext^{p+q}(M,M))_p.
$$
The associated graded modules in the last line arise from the  filtration on $H(\mathbf{L}^\bullet)$ induced form the one of $\mathbf{L}^\bullet$.
We now analyze the cancellations from the page $\mathbf{E}_r$ to the page $\mathbf{E}_{r+1}$, for  each  $r\geq 1$.
The vector spaces in  $\mathbf{E}_{r+1}$ are constructed as cohomology groups of the complexes in $\mathbf{E}_r$
$$
\cdots \xrightarrow{d_r} \mathbf{E}_{r}^{p+r,q-r-1} \xrightarrow{d_r}
\mathbf{E}_{r}^{p,q} \xrightarrow{d_r} \mathbf{E}_r^{p-r,q+r+1} \xrightarrow{d_r}\cdots.
$$
Since the differential  $d_r$ is $\deg$-homogeneous of degree 0 and  increases the cohomological degree by one,
it follows that  $\mathbf{E}_{r+1}$ is obtained from $\mathbf{E}_r$ by canceling $\deg$-graded split-exact complexes $0 \to \kk \to \kk \to 0$,
that is, pairs of copies of $\kk$ concentrated in the same internal degree $\deg(\cdot)$ and in consecutive cohomological degrees.
By combining the cancellations from each page $\mathbf{E}_r$ with $r\geq 1$, the latter statement holds for the cancellations between the page $\mathbf{E}_1$ and the page $\mathbf{E}_\infty$, and we obtain the desired statement.
\end{proof}

We illustrate Lemma \ref{LemmaCancellations} with the following  example.

\begin{example}\label{ExampleCancellations}
Let  $S = \mathbb{C}[x_1, x_3, x_3, x_4]$ and let $F = S \varepsilon_1 \oplus  S \varepsilon_2$ be a free module of rank $2$, with $\deg(\varepsilon_1) = \deg (\varepsilon_2) =0$.
Equip $F$ with the monomial order $\preceq$ defined by 
$$
\mu \varepsilon_i \preceq \nu \varepsilon_j
\quad 
\text{if}
\quad i<j \quad \text{or} \quad i=j \quad \text{and} \quad \mu <_{\text{rev}} \nu,
$$
where $\mu, \nu \in S$ are monomials and $<_{\text{rev}}$ denotes the graded reverse lexicographic order.
Let
$$K = \langle x_1 \varepsilon_1, \, x_2 \varepsilon_1 - x_1 \varepsilon_2, \,
x_3 \varepsilon_1-x_2 \varepsilon_2, \,  x_4 \varepsilon_1 - x_3 \varepsilon_2,\,
x_4 \varepsilon_2\rangle \subseteq F,$$ 
then, $K$ is a graded submodule, and $M = F/K$ is a finite  $S$-module of length 5.
The monomial initial submodule of $K$ is 
$$
K' = \langle x_1 \varepsilon_1, \,
 x_1 \varepsilon_2, \, x_2 \varepsilon_2, \, x_3 \varepsilon_2, \, x_4 \varepsilon_2,\,
  x_2^2 \varepsilon_1,\, x_2x_3 \varepsilon_1,\, x_3^2 \varepsilon_1,\,
  x_2x_4 \varepsilon_1,\, x_3x_4 \varepsilon_1,\, x_4^2 \varepsilon_1  
  \rangle \subseteq F.
  $$
We compute the dimension tables of the Ext modules:

\begin{multicols}{2}
\begin{center}
\begin{tabular}{|cc|c|c|c|c|c|}
\hline
& $i$ & 0 & 1 & 2 & 3 & 4\\
$j$ &  &  & & & & \\
 \hline
$-5$ &  &  & & &  & 6\\
 \hline
$-4$ &  &  & & & 12 & 1 \\
 \hline
$-3$ &  &  &  &  & 10 &  \\
 \hline
$-2$ &  &  &  & 30 &  & \\
 \hline
$-1$ &  &  & 10 & & & \\
 \hline
$0$ &  & 1 & 12 & & & \\
 \hline
$1$ &  & 6 & & & & \\
\hline
\end{tabular}
$$
\dim_\kk \left[ \Ext^i_S(M,M)\right]_j
$$

\end{center}

\columnbreak

\begin{center}
\begin{tabular}{|cc|c|c|c|c|c|}
\hline
& $i$ & 0 & 1 & 2 & 3 & 4 \\
$j$ &  &  & & & & \\
 \hline
$-5$ &  &  & & &  & 6\\
 \hline
$-4$ &  &  & & & 14 & 3\\
 \hline
$-3$ &  &  &  & {14} & {24} & \\
 \hline
$-2$ &  &  & {6} & {42} & 6 &  \\
 \hline
$-1$ &  &  & {24} & {14} & & \\
 \hline
$0$ &  & {3} & {14} & & & \\
 \hline
$1$ &  & 6 & & & & \\
\hline
\end{tabular}
$$
\dim_\kk \left[ \Ext^i_S(M',M')\right]_j
$$
\end{center}

\end{multicols}
\noindent
and determine the consecutive cancellations of Lemma \ref{LemmaCancellations}:
$$
{\delta_{0,0}} = 2,\quad {\delta_{1,-1}}=14,\quad {\delta_{1,-2}}=6,\quad \delta_{2,-2}=6, \quad \delta_{2,-3}=14, \quad {\delta_{4,-4}} = 2.
$$
\end{example}

\begin{remark}
Related ``consecutive cancellation" phenomena occur  for Betti tables
\cite{Peeva,RossiSharifan} 
and for Tor modules \cite{Sammartano}.
\end{remark}

Next, we need the following duality statement for Ext modules.

\begin{lemma}[Duality]\label{LemmaDuality}
Let $S = \kk[x_1, \ldots, x_n]$  and let $M, N$ be two graded $S$-modules of finite length.
Then, we have an isomorphism of $\mathbb{Z}$-graded vector spaces
$$
\Ext^i_S(M,N)^\vee \cong \Ext^{n-i}_S(N,M)(-n)
$$
where $-^\vee$ denotes the $\kk$-dual.
\end{lemma}
\begin{proof}
This can be proved as in \cite[Lemma 2.2]{RS}.
\end{proof}

\begin{notation}
For the rest of the paper, 
we introduce the following  notation:
$$
hom(\cdot, \cdot) := \dim_\kk \Hom(\cdot, \cdot)
\quad \text{and}\quad
ext^i(\cdot, \cdot) := \dim_\kk \Ext^i(\cdot, \cdot).
$$
\end{notation}

We observe another parity phenomenon for finite subschemes of $\AA^3$.

\begin{lemma}\label{LemmaHomIJ}
Let  $S = \kk[x_1, x_2, x_3]$
and $I,J \subseteq S$ be ideals of finite colength.
Then,
$$
hom(I,S/J) + hom(J,S/I) \equiv \dim_\kk(S/I)+\dim_\kk(S,J) \pmod{2}.
$$
\end{lemma}

\begin{proof}
From the  exact sequences
$$
0 \to \Hom(S/I,S/J) \to \Hom(S,S/J)  \to \Hom(I, S/J) \to \Ext^1(S/I,S/J) \to 0
$$
and
$$
0 \to \Hom(S/J,S/I) \to \Hom(S,S/I)  \to \Hom(J, S/I) \to \Ext^1(S/J,S/I) \to 0
$$
we obtain the equation
\begin{multline*}\label{EqTwoExactSeq}
hom(I,S/J)+hom(J,S/I) - hom(S,S/J) - hom(S,S/I) 
=
\\
ext^1(S/I,S/J) + ext^1(S/J,S/I) 
-hom(S/I,S/J)-hom(S/J,S/I).
\end{multline*}
Now, the desired  conclusion 
follows, since 
\begin{align*}
& \quad hom(I,S/J)+hom(J,S/I) - \dim_\kk(S/J) - \dim_\kk(S/I) 
\\
 = & \quad
hom(I,S/J)+hom(J,S/I) - hom(S,S/J) - hom(S,S/I)  
\\
= & \quad
ext^1(S/I,S/J) + ext^1(S/J,S/I) 
-hom(S/I,S/J)-hom(S/J,S/I)
\\
= & \quad ext^2(S/J,S/I) + ext^1(S/J,S/I) 
-ext^3(S/J,S/I)-ext^0(S/J,S/I) && \text{ by Lemma \ref{LemmaDuality}}
\\
\equiv & \quad ext^2(S/J,S/I) - ext^1(S/J,S/I) 
-ext^3(S/J,S/I)+ext^0(S/J,S/I) 
\\
= & \quad 0 \pmod{2},
\end{align*}
where the last equation is  the vanishing of the Euler characteristic 
$
\sum_i (-1)^i ext^i(S/J,S/I).
$
\end{proof}

As a consequence,
 Question \ref{QueQuot} has an affirmative answer in the case of  monomial submodules.
Recall that a submodule of a free module is called \emph{monomial} if it is a direct sum of monomial ideals. 

\begin{cor}\label{CorollaryMonomialCase}
Let  $S = \kk[x_1, x_2, x_3]$
and $F$ a free $S$-module of rank $r$.
Let $M = F/K$ be an $S$-module of finite length $d$.
If $K\subseteq F$ is a monomial submodule, then 
$hom(K,M) \equiv rd \pmod{2}.$
\end{cor}

\begin{proof}
Let
$F = \bigoplus_{i=1}^r S \varepsilon_i$ be the free module.
Then, $K = \bigoplus_{i=1}^r I_i \varepsilon_i$, where each $I_i \subseteq S$ is a monomial ideal of finite colength. 
Let $d_i = \dim_\kk(S/I_i)$, then we have $d = \sum_{i=1}^r d_i$.
We have
\begin{align*}
hom(K,M) &= hom\left(\bigoplus_{i=1}^r I_i \,, \,\bigoplus_{j=1}^r S/I_j\right)
= \sum_{i,j=1}^r hom\left(I_i\, , S/I_j\right)\\
&=
\sum_{i=1}^r  hom(I_i,S/I_i) + \sum_{i<j}\big(  hom(I_i,S/I_j) + hom(I_j,S/I_i)\big)\\
& \equiv \sum_{i=1}^r d_i + \sum_{i<j}( d_i+d_j) = \sum_{i=1}^r rd_i = rd \pmod{2},
\end{align*}
where the congruence holds by the monomial case of Conjecture \ref{ConjOP} \cite{MNOP} and by  Lemma \ref{LemmaHomIJ}.
\end{proof}

We are ready to prove the main theorem of this paper,
namely, we settle Question \ref{QueQuot} for points $z \in \mathrm{Quot}_{\AA^3}^d(\mathcal{O}^{\oplus r})$
that parametrize \emph{graded} $S$-modules $M$.

\begin{thm} \label{TheoremParityModules}
Let  $S = \kk[x_1, x_2, x_3]$
and $F$ a free graded $S$-module of rank $r$.
Let $M = F/K$ be a graded $S$-module of finite length $d$.
Then,
$$hom(K,M) \equiv rd \pmod{2}.$$
\end{thm}

\begin{proof} 
Let $K'\subseteq F$ be the monomial initial submodule with respect to a monomial order $\preceq$ on $F$ that refines the grading, and let $M'=F/K'$.
We denote $E_i = \Ext_S^i(M,M)$ and $E'_i = \Ext_S^i(M',M')$,
for $i = 0, 1, 2, 3$.
Observe that Lemma \ref{LemmaDuality} implies  that $\dim_\kk [E_i]_j = \dim_\kk [E_{3-i}]_{-j-3}$ for all $i,j$, 
and similarly for $E'_i$.

Recall that the tangent space to the point parametrizing $M$, respectively $M'$,
is $\Hom(K,M)$, respectively $\Hom(K',M')$.\footnote{
We warn the reader that, unlike the case of $\Hilb^d(\AA^3)$,
the tangent space $\Hom(K,M)$ is typically \emph{not} isomorphic to $\Ext^1(M,M)$ or $\Ext^1(K,K)$.}
From the exact sequence 
$$
0 \to \Hom(M,M) \to \Hom(F,M) \to \Hom(K,M) \to \Ext^1(M,M) \to 0
$$
we obtain the equation 
$$
hom(K,M) = ext^1(M,M) + hom(F,M) - hom(M,M)   =  ext^1(M,M) + rd - ext^0(M,M).
$$
Likewise, we have 
$
hom(K',M') =  ext^1(M',M') + rd - ext^0(M',M').  
$
Applying Lemma \ref{LemmaCancellations},
we obtain
\begin{align*}
hom(K',M')-hom(K,M) &= \big(ext^1(M',M')-ext^1(M,M)\big)  - \big(ext^0(M',M')-ext^0(M,M)\big)\\
&= \sum_j (\delta_{1,j} +\delta_{0,j}) - \sum_j \delta_{0,j}\\
&= \sum_j \delta_{1,j}.
\end{align*}

Applying both Lemmas \ref{LemmaCancellations} and   \ref{LemmaDuality}, we obtain for all $j$
\begin{equation}\label{EqDelta20}
\delta_{2,j} = \dim_\kk [E'_3]_j - \dim_\kk [E_3]_j = \dim_\kk [E'_0]_{-j-3} - \dim_\kk [E_0]_{-j-3} = \delta_{0,-j-3}.
\end{equation}
Moreover, we have
\begin{align*}
\delta_{1,j} &= \dim_\kk [E'_1]_j - \dim_\kk [E_1]_j - \delta_{0,j}  && \text{by Lemma \ref{LemmaCancellations}}
\\
&= \dim_\kk [E'_2]_{-j-3} - \dim_\kk [E_2]_{-j-3} - \delta_{0,j} && \text{by Lemma \ref{LemmaDuality}}
\\
&= \dim_\kk [E'_2]_{-j-3} - \dim_\kk [E_2]_{-j-3} - \delta_{2,-j-3} && \text{by \eqref{EqDelta20} }
\\
&= 
\delta_{1,-j-3} && \text{by Lemma \ref{LemmaCancellations}}.
\end{align*}
In particular, we have  
$\sum_{j \text{ even}} \delta_{1,j} = \sum_{j \text{ odd}} \delta_{1,j}$,
so $\sum_{j } \delta_{1,j}$ is an even integer.

Finally, by Corollary \ref{CorollaryMonomialCase} we have
$$
hom(K,M) = hom(K',M') - \sum_{j } \delta_{1,j}  \equiv hom(K',M') \equiv rd \pmod{2},
$$
and this concludes the proof.
\end{proof}

\begin{remark}\label{RemarkGradings}
The same argument settles Question \ref{QueQuot} in a slightly greater generality,
namely, for modules $M$ that are graded with respect  to a grading of $S$ taking values in a torsion-free abelian group  such that $\deg(x_1)+\deg(x_2)+\deg(x_3)$ is not divisible by 2.
For example, one can take the $\Z$-grading defined by $\deg(x_1) = 4, \deg(x_2) = 5, \deg(x_3) = 6$, 
or the $\Z^2$-grading defined by $\deg(x_1) = (1,0)$, and $\deg(x_2) =  \deg(x_3) = (0,1)$.
It turns out that this condition is sharp, 
as the following example of Giovenzana-Giovenzana-Graffeo-Lella shows.
\end{remark}

\begin{example}\label{ExampleCancellationsCounterexample}
We analyze the counterexample of \cite{GGGL} from the point of view of the proof of Theorem \ref{TheoremParityModules}.
Let $S = \mathbb{C}[x_1,x_2,x_3]$ and 
let $I = \big( x_1 + (x_2,x_3)^2)\big)^2 + \big(x_1x_2-x_3^3\big) \subseteq S$.
Then, $S/I$ has length 12, but the corresponding point $z \in \Hilb^{12}(\AA^3)$ has tangent space of dimension 45.
The ideal $I$ is homogeneous with respect to the grading $\deg(x_1) = 2, \deg(x_2)=\deg(x_3) =1$.
The monomial ideal $I' = \big( x_1 + (x_2,x_3)^2)\big)^2 + \big(x_1x_2\big)$ is 
the initial ideal of $I$ with respect to the graded lexicographic monomial order in $S$,
and the corresponding point $z' \in \Hilb^{12}(\AA^3)$ has tangent space of dimension 48.

Letting $M = S/I, M'=S/I'$, 
we compute the dimension tables of the Ext modules

\begin{multicols}{2}
\begin{center}
\begin{tabular}{|cc|c|c|c|c|}
\hline
& $i$ & 0 & 1 & 2 & 3\\
$j$ &  &  & & &  \\
 \hline
$-7$ &  &  & & & 5 \\
 \hline
$-6$ &  &  & & & 4 \\
 \hline
$-5$ &  &  & & & 2 \\
 \hline
$-4$ &  &  & & 5 & 1 \\
 \hline
$-3$ &  &  &  & 39 &   \\
 \hline
$-2$ &  &  & 1 & 1 &  \\
 \hline
$-1$ &  &  & 39 & &  \\
 \hline
$0$ &  & 1 & 5 & &  \\
 \hline
$1$ &  & 2 & & &  \\
\hline 
$2$ &  & 4 & & &  \\
\hline 
$3$ &  & 5 & & &  \\
\hline
\end{tabular}
$$
\dim_\kk \left[ \Ext^i_S(M,M)\right]_j
$$

\end{center}

\columnbreak

\begin{center}
\begin{tabular}{|cc|c|c|c|c|}
\hline
& $i$ & 0 & 1 & 2 & 3\\
$j$ &  &  & & &  \\
 \hline
$-7$ &  &  & & & 5 \\
 \hline
$-6$ &  &  & & & 4 \\
 \hline
$-5$ &  &  & & & 2 \\
 \hline
$-4$ &  &  & & 5 & 1 \\
 \hline
$-3$ &  &  &  & 39 &  \\
 \hline
$-2$ &  &  & {4} & {4} &  \\
 \hline
$-1$ &  &  & 39 & &  \\
 \hline
$0$ &  & 1 & 5 & &  \\
 \hline
$1$ &  & 2 & & &  \\
\hline 
$2$ &  & 4 & & &  \\
\hline 
$3$ &  & 5 & & &  \\
\hline
\end{tabular}

$$
\dim_\kk \left[ \Ext^i_S(M',M')\right]_j
$$
\end{center}

\end{multicols}
\noindent
and find that the only consecutive cancellation of Lemma \ref{LemmaCancellations} is ${\delta_{1,-2}}=3$.
In terms of the pairing established in the proof of Theorem \ref{TheoremParityModules}, 
since $\deg(x_1)+\deg(x_2)+\deg(x_3) = 4$,
we see that $\delta_{1,-2} $ is paired with itself, because, for $j = -2$, we have $-j-4 = j$.
Thus, the difference $\dim T_{z'}-\dim T_{z} = \delta_{1,-2}$ is odd.
\end{example}

\begin{remark}\label{RemarkExamplesHilbN}
We point out that no analogue of Conjecture \ref{ConjOP}  holds for $\Hilb^d(\AA^n)$ when $n \geq 4$.
Indeed,
for sufficiently large $d$, the tangent spaces to $\Hilb^d(\AA^n)$
 attain nearly all dimensions in a large interval of integers.
For example, a computation with \cite{M2} shows that 
 the dimensions of tangent spaces to points  in $\Hilb^{10}(\AA^4)$
 parametrizing  monomial ideals are 
$$
40,\, 41,\, 46,\, 48,\, 49,\, 50,\, 51,\, 52,\, 53,\, 55,\, 56,\, 57,\, 60,\, 61,\, 65,\, 66,\, 69,\, 70,\, 76,\, 80.
$$
Similarly,
the dimensions of tangent spaces to 
points  in $\Hilb^{10}(\AA^5)$
 parametrizing 
 monomial ideals are 
$$
50,\, 51,\, 56,\, 58,\, 59,\, 60,\, 61,\, 62,\, 63,\, 65,\, 66,\, 67,\, 68,\, 70,\, 71,\, 73,
$$
$$
 75,\, 76,\, 78,\, 79,\, 80,\, 81,\, 82,\, 83,\, 85,\, 86,\, 90,\, 95,\, 98,\, 101,\, 104.
$$
\end{remark}

Clearly, a   phenomenon analogous to Remark \ref{RemarkExamplesHilbN}
also occurs  for Quot schemes of $\AA^n$ with $n\geq 4$.
However, surprisingly, it also occurs for $n=2$.
While the Hilbert scheme $\Hilb^d(\AA^2)$ is smooth and irreducible of dimension $2d$, 
the Quot scheme 
 $\mathrm{Quot}^{d}_{\AA^2}(\mathcal{O}^{\oplus r})$ 
 is also irreducible, but usually  singular, if $r \geq 2$ \cite{EL,GL}. 
The tangent spaces to  $\mathrm{Quot}^{d}_{\AA^2}(\mathcal{O}^{\oplus r})$ 
 attain nearly all dimensions in an  interval of integers.

\begin{example}\label{ExampleQuotA2}
The dimensions of tangent spaces to 
points  in $\mathrm{Quot}^{20}_{\AA^2}(\mathcal{O}^{\oplus 2})$ 
 parametrizing 
 monomial submodules are 
$$
60,\, 62,\, 63,\, 64,\, 65,\, 66,\, 67,\,  68,\, 69,\, 70,\, 71,\, 72,\, 73,\, 74,\, 75,\, 76,\, 77,\, 78,\, 79,\, 80.
$$
\end{example}

\section*{Acknowledgments}
The authors would like to thank Rahul Pandharipande for some helpful correspondence.
Some computations were performed with the software Macaulay2 \cite{M2}.

\section*{Funding}
Part of this work was done at the Institute of Mathematics of the Polish Academy of Sciences, during the Thematic Research Programme 
``Tensors: geometry, complexity and quantum entanglement'',
with support from University of Warsaw, Excellence Initiative – Research University and the Simons Foundation Award No. 663281.
Alessio Sammartano was partially supported by 
 PRIN 2020355B8Y “Squarefree Gr\"obner degenerations, special varieties and related topics”.

\end{document}